\documentclass{amsart}
\usepackage{amsmath}
\newtheorem{theorem}{Theorem}[section]
\newtheorem{corollary}[theorem]{Corollary}
\newtheorem{proposition}[theorem]{Proposition}
\newtheorem{lemma}[theorem]{Lemma}
\theoremstyle{definition}

\newtheorem{remark}[theorem]{Remark}

\def\irr{\mbox{\rm Irr}}

\def\C{{\mathbb C}}

\begin{document}

\overfullrule=0pt

\title{The  depth of subgroups of Suzuki groups}

\author{L.~H\'ethelyi}
\address{\'Obuda University, H-1034 Budapest, B\'ecsi \'ut 96/B.}
\email{fobaba@t-online.hu}

\author{E.~Horv\'ath}
\address{Department of Algebra, Budapest University of Technology and Economics, H-1521 Budapest, M\H uegyetem rkp. 3--9.}
\email{he@math.bme.hu}

\author{F.~Pet\'enyi}
\address{Department of Algebra, Budapest University of Technology and Economics,
 H-1521 Budapest, M\H uegyetem kp. 3--9.}
\email{petenyi.franciska@gmail.com}
\abstract{%
We determine the combinatorial depth of certain subgroups of simple Suzuki
 groups $Sz(q)$,
among others,
 the depth of their maximal subgroups. We apply these results to determine
the ordinary depth of these subgroups.
}
}

\maketitle

\section{The combinatorial depth and ordinary depth}

In this paper we shall study the  depth of subgroups of Suzuki groups.
The notion of combinatorial depth was defined in \cite{BDK}. The related concept
of ordinary depth has its origins in von-Neumann algebras, see \cite{GHJ}.

 First we remind the reader to some of the basic definitions,
 results and notations  on combinatorial depth, see \cite{BDK}.\\

Let $G$ be a finite group, $H$ a subgroup of $G$.
To define the combinatorial depth of $H$ in $G$,  we need some facts about bisets.

 Let $J,K,L$ be finite groups, let  $X$ be  a $(J,K)$-biset, let $Y$ be a $(K,L)$-biset. Then $X\times Y $ will be a $(J,L)$-biset. The group $K$ also acts on $X\times Y$
by $k\cdot (x,y):=(xk^{-1},ky)$ for $x\in X,y\in Y, k\in K$.
The set of $K$-orbits of this action is denoted by $X\times_K Y$.
This set also inherits a $(J,L)$-biset structure. Let 
$\Theta_1(H,G)$ be the
$(G,G)$-biset $G$, using left and right multiplication of $G$ on itself.
Let $\Theta_{i+1}(H,G):=\Theta_i(H,G)\times_H G$ for $i\geq 1$.
We denote by $\Theta'_i(H,G)$ the set $\Theta_i(H,G)$ considered as a $(H,H)$-biset.
We denote by $\Theta^{l}_i(H,G)$ and $\Theta^{r}_i(H,G)$, the set $\Theta_i(H,G)$
considered as a 
 $(H,G)$-biset or $(G,H)$-biset, respectively. 
 Furthermore, $\Theta_0'(H,G):=H$, as a 
$(H,H)$-biset.
The subgroup $H$ is said to be of combinatorial depth $2i$ in $G$ if there exist
 natural numbers $a_1,a_2$ such that  $\Theta^r_{i+1}(H,G)$  
 is a direct summand of $a_1\cdot \Theta^r_i(H,G)$ 
for some $i\geq 1$ and $\Theta^l_{i+1}$ is a direct summand of $a_2\cdot \Theta^l_{i}(H,G)$, respectively.
Moreover $H$ has combinatorial depth $2i+1$ if ${\Theta^{'}}_{i+1}(H,G)$ is a direct
summand of $a \cdot {\Theta^{'}}_i(H,G)$ for some natural number $a$.
The minimal combinatorial depth $d_c(H,G)$ is the smallest positive integer  $i$ 
such that
$H$ has combinatorial depth $i$ in $G$. This  number is well defined.

 Let us denote by
 $H^x:=x^{-1}Hx$ and $H_{x_1,\ldots ,x_n}:=H\cap H^{x_1}\cap \ldots \cap H^{x_n}$, for
$x_1,\ldots ,x_n\in G$.
Let  ${\mathcal{U}}_i:={\mathcal{U}}_i(H,G):=\{ H_{x_1,\ldots x_i} : x_1,\ldots x_i \in G \}$  and ${\mathcal{U}}_{\infty}:={\mathcal{U}}_{\infty}(H,G):=
\cup_{i\geq 0}{\mathcal{U}}_i$, where ${\mathcal{U}}_0:=\{H\}$.
In \cite{BDK} the following characterization of combinatorial depth $d_c(H,G)$ 
is proved:\\

\noindent
\begin{theorem}\cite[Thm~3.9]{BDK}\label{BDK3.9}
Let $H$ be a subgroup of the finite group $G$ and let  $i\geq 1$.
Then:
\begin{itemize}
\item [(i)] $d_c(H,G)\leq 2i$ $\leftrightarrow$  ${\mathcal{U}}_{i-1}={\mathcal{U}}_i$
$\leftrightarrow$ ${\mathcal{U}}_{i-1}={\mathcal{U}}_{\infty}$ $\leftrightarrow$  for any
$x_1,\ldots ,x_i\in G$, there exist $y_1,\ldots ,y_{i-1}\in G$ with
$H_{x_1,\ldots ,x_i}=H_{y_1,\ldots ,y_{i-1}}$.\\
\item [(ii)] Let $i>1$. Then  $d_c(H,G)\leq 2i-1 $ $\leftrightarrow $ for any $x_1,\ldots x_i \in G$
there exist $y_1,\ldots y_{i-1}\in G$ with 
$H_{x_1,\ldots x_i}=H_{y_1,\ldots ,y_{i-1}}$ and $x_1hx_1^{-1}=y_1hy_1^{-1}$ for all
$h\in H_{x_1,\ldots x_i}$.\\
\item [(iii)] $d_c(H,G)=1$ $\leftrightarrow $  for any $x\in G$ there exists a $y\in H$ with
$xhx^{-1}=yhy^{-1}$ for  all $h\in H$ $\leftrightarrow$ $G=HC_G(H)$.
\end{itemize}
\end{theorem}

In \cite{BDK}  the {\it depth of a partially ordered set} $X$ is introduced,
 as the length of a largest chain in $X$, and $\delta $ is defined as the depth
of $\mathcal{U}_{\infty}$$(H,G)$. The symbol $\delta_*$  denotes the smallest positive integer $k$
such that $Core_G(H)$ can be written as the intersection of $k$ conjugates of $H$,
and $\delta ^*$ denotes the smallest positive  integer $k$ such that the intersection
of any $k$ distinct conjugates of $H$ is equal to $Core_G(H)$.

\noindent
\begin{theorem}\cite[Thm~3.11]{BDK}\label{BDK3.11}
Let $G$ be a finite group, $H$ a subgroup of $G$ and $K:=Core_G(H)$.Then:
\begin{itemize}
\item [(a)]$ 2\delta_{*}-1\leq d_c(H,G)\leq 2\delta $
\item [(b)] $\delta_{*}\leq \delta \leq \delta^{*}\leq |G:N_G(H)|$
\item [(c)] If $\delta_{*}=\delta $ and $K\leq Z(G)$ then $d_c(H,G)=2\delta -1$.
\end{itemize}
\end{theorem}

Now let us remind the reader to the notion of \emph{ordinary depth} of a subgroup
 $H$ in the finite group $G$.
We say that the {\em depth of the group algebra inclusion } ${\C}H \subseteq {\C}G$ is $2n$
if ${\C}G\otimes_{{\C}H}\cdots\otimes_{{\C}H}{{\C}G}$ ($n+1$-times ${\C}G$)
 is isomorphic to a direct summand of 
$\oplus_{i=1}^a{\C}G\otimes_{{\C}H}\cdots \otimes_{{\C}H}{\C}G$ ($n$ times ${\C}G$)
as ${\C}G-{\C}H$-bimodules for some positive integer $a$.

Furthermore ${\C}H$ is said to have depth $2n+1$ in ${\C}G$ if the same assertion holds for  ${\C}H-{\C}H$-bimodules. Finally ${\C}H$ has depth $1$ in ${\C}G$ if
${\C}G$ is isomorphis to a direct summand of $\oplus_{i=1}^a{\C}H$ as
 ${\C}H-{\C}H$ bimodules.
The minimal depth of group algebra inclusion ${\C}H\subseteq {\C}G$ is called
(minimal) ordinary depth of $H$ in $G$, which we denote by $d(H,G)$.
This is well defined.

 This depth can be obtained from
the so called inclusion matrix $M$. If $\chi_1,\ldots ,\chi_s$ are
all irreducible characters of $G$ and $\psi_1,\ldots ,\psi_r$ are all irreducible characters of $H$, then $m_{i,j}:=(\psi_i^G,\chi_j)$.
The "powers" of $M$ are defined by $M^{2l}:=M^{2l-1}M^T$ and $M^{2l+1}:=M^{2l}M$
The odinary depth $d(H,G)$ can be obtained as the smallest integer $n$
such that $M^{n+1}\leq aM^{n-1}$. This is well defined.
The  results on characters in \cite{BKK}   help
to determine $d(H,G)$. Two irreducible characters $\alpha,\beta \in \irr(H)$ are related,
$\alpha \sim \beta$, if they are constituents of some $\chi_H$, for 
 $\chi \in \irr(G)$. The distance $d(\alpha,\beta)=m$
is the smallest integer $m$ such that there is a chain of irreducible
 characters of $H$ such that $\alpha=\psi_0\sim \psi_1\ldots \sim
\psi_m=\beta $. If there is no such chain then 
$d(\alpha,\beta)=-\infty $ and if $\alpha=\beta $ then the distance is zero.
 If $X$ is the set of irreducible constituents of $\chi_H$ then 
$m(\chi):=\max_{\alpha \in \irr(H)}\min_{\psi \in  X} d(\alpha,\psi)$.
The following results from \cite{BKK}
will be useful.

\begin{theorem}\cite[Thm~3.9, Thm~ 3.13]{BKK}\label{distance}
\begin{itemize}
\item[(i)] Let $m\geq 1$. Then $H$ has ordinary depth $\leq 2m+1$ in $G$
if and only if the distance between two irreducible characters of $H$ is at most $m$.
\item[ (ii)] Let $m\geq 2$. Then $H$ has ordinary depth $\leq 2m$ in $G$
if and only if $m(\chi)\leq m-1$ for all $\chi \in Irr(G)$.
\end{itemize}
\end{theorem}

\begin{theorem}\cite[Thm~6.9]{BKK}\label{inter}
Suppose that  $H$ is a subgroup of a finite group $G$ and $N=Core_G(H)$
is the intersection  of $m$ conjugates of $H$. Then $H$ has ordinary depth
$\leq 2m$ in $G$. If $N\leq Z(G)$ then $d(H,G)\leq 2m-1$.
\end{theorem}

Similarly as in the case of combinatorial depth, $d(H,G)\le 2$ if and only if $H$ is normal in $G$. The ordinary depth does not change if we replace the complex number field by any field of characteristic zero.
By \cite[Remark 4.5]{BDK}  the ordinary depth $d(H,G)$ is bounded from above by
the combinatorial depth $d_c(H,G)$.

For papers  computing combinatorial or ordinary
 depth in some other groups, see
 \cite{Frit1},\cite{Frit2}\cite{FKR}, where $PSL(2,q)$ and $S_n$ are considered.

\section{The Suzuki groups}

Let $q$ always denote an odd power of $2$, say $q=2^{2m+1}$, where $m>0$.
The  Suzuki groups can be defined in various ways. They are twisted simple 
groups
of Lie type ${}^2B_2(q)$, see e.g. \cite[p. 43]{CSM}.
(In fact $Sz(2)={}^2B_2(2)$ is also defined, but it is not simple. It is a Frobenius group
of order $20$.)

As permutation groups, the Suzuki groups belong to the class of Zassenhaus groups,
i.e. doubly transitive permutation groups 
without any regular normal subgroups, where any non-identity element has at most
two fixed points, and there are elements with exactly two fixed points.
According to a well-known theorem of Feit, see \cite[Thm. 6.1 in Ch.~XI]{H-Bl_III},
the degree of a Zassenhaus group is always $q_1+1$, where $q_1$ is a prime power.
For even $q_1$,  there are two series of such groups, $PSL(2,2^m)$ and
$Sz(2^{2m+1})$. Suzuki proved that these are all the Zassenhaus groups for even $q_1$.
See e.~g.~ \cite{SuzI,SuzII}.

The Suzuki groups can also be defined  as subgroups of $GL(4,q)$.
See e.g. \cite[Ch XI. p. 182]{H-Bl_III}.

Let $\pi $ be the unique automorphism of the field $K=GF(q)$ with $x\pi^2=x^2$ for
every $x\in K$.

Then $Sz(q):=\langle S(a,b),M(\lambda), T |
a,b\in GF(q), \lambda \in GF(q)^{\times}\rangle$, where

\[S(a,b):=
\left(\begin{array}{rrrr}
1 & 0 & 0 & 0 \\
a & 1 & 0 & 0 \\
b & a\pi& 1& 0 \\
a^2(a\pi) + ab+ b\pi & a(a\pi) + b & a & 1
\end{array}\right)\]

\[M(\lambda):=
\begin{pmatrix} \lambda^{1+2^m} & 0 & 0 & 0\\
	       0 & \lambda^{2^m} & 0 & 0\\
	       0 & 0 & \lambda^{-2^m} & 0\\
	       0 & 0 & 0 & \lambda^{-1-2^m}
\end{pmatrix}\] and

\[T:=
\begin{pmatrix} 0 & 0 & 0 & 1 \\
      0 & 0 & 1 & 0 \\
      0 & 1 & 0 & 0\\
      1 & 0 & 0 & 0
\end{pmatrix}\]

The $q^2+1$ points $\{p_{\infty}, p(x,y)| x,y\in K\}$, where
$p(x,y)=[xy+(x\pi)x^2+y\pi,y,x,1]$,  form an ovoid $\mathcal O$, in the projective space $P(3,q)$.
It means that no $3$ points of $\mathcal O$ are on one line.
The group $Sz(q)$ is the group of all projective collineations on $P(3,q)$ which
leave invariant the set of $q^2+1$ points of this ovoid. This action is faithful and $Sz(q)$ as a permutation group on these points is a Zassenhaus group of order $(q^2+1)q^2(q-1)$.

Let $G=Sz(q)$. Elements which only fix the  point $p_{\infty}$
constitute the nontrivial elements of  a Sylow $2$-subgroup
$F$ of $G$. The pointwise  stabilizer of
the set  $\{p(0,0),p_{\infty}\}$ is a cyclic subgroup $H$ of order $q-1$, which normalizes $F$.

Thus the stabilizer of $p_{\infty}$ is $FH=N_G(F)$ which is a Frobenius group.

The normalizer $N_G(H)$ is a dihedral group of order $2(q-1)$.

We note that the numbers $q^2+1,q^2,q-1$ are pairwise relatively prime.

For odd prime $p$ the Sylow $p$-subgroups of $G$ are cyclic.
$Sz(q)$ is a $CN$-group, i.e.  every
nontrivial element has a nilpotent centralizer.

For the proofs of the above statements, see e.g.
\cite[Ch.~XI p.~182--194]{H-Bl_III}.

The Sylow $2$-subgroup  $F$ of $G$ is a  Suzuki
$2$-group. This means that it is a non-Abelian $2$-group, having 
more than one involution,
and having a solvable group of automorphisms which permutes the set of involutions
of $F$ transitively.

See  \cite[p.~299]{H-Bl_II} for details.
It is a class $2$ group  of order $q^2$ and exponent
$4$. Moreover its center $Z(F)$ is of order $q$.
The involutions in $F$ together   with the identity element constitute
$Z(F)$. 
The subgroup $H$ acts  sharply 1-transitively on the involutions of $F$.
The centralizer in $G$ of every nontrivial element of $F$ is a  subgroup of $F$.

\section{Subgroups of Suzuki groups}

We collect some of the basic facts about the subgroups of Suzuki groups.
For these  and their proofs, see \cite[Thm.~4.12]{Lun}                 
and \cite[Thm~3.10 in Ch.~XI]{H-Bl_III}.

\noindent
\begin{theorem}[Suzuki]\label{Suzuki}
Let $G=Sz(q)$, where $q=2^{2m+1}$, for some positive integer $m$. Then $G$
has the following subgroups:
\begin{itemize}
\item [1.] the Hall subgroup  $N_G(F)=FH$, which is a Frobenius group
of order $q^2(q-1)$.
\item [2.] the dihedral group $B_0=N_G(H)$ of order $2(q-1)$.
\item [3.] the  cyclic Hall subgroups $A_1,A_2$ of orders $q+2r+1,q-2r+1$, respectively, where
$r=2^m$ and $|A_1||A_2|=q^2+1$.
\item [4.] the Frobenius subgroups $B_1=N_G(A_1), B_2=N_G(A_2)$ of orders $4|A_1|,4|A_2|$, respectively.
\item [5.] the  subgroups of form $Sz(s)$, where $s$ is an odd power of $2$, $s\geq 8$, and
$q=s^n$ for some positive integer $n$. Moreover, for every odd $2$-power
$s$, where $s^n=q$ for some positive
integer $n$,
there exists a subgroup isomorphic to $Sz(s)$.  
\item [6.] Subgroups (and their conjugates) of  the above groups.

\end{itemize}
\end{theorem}

\noindent
\begin{theorem}\label{info}
Let $q=2^{2m+1}$, $m>0$, $r=2^m$ and $G=Sz(q)$.
\begin{itemize}
\item [a)] Let $i\in \{1,2\}$ and let  $u_i\in A_i$, $u_i\ne 1$. Then $C_G(u_i)=A_i$.
If
$B_i=N_G(A_i)$  then 
$B_i=\langle A_i,t_i\rangle$, where $t_i$ is an element of order $4$,
and $u^{t_i}=u^q$, for all $u\in A_i$.
Moreover, $N_G(A_i)$ is a Frobenius group with kernel $A_i$.
\item [b)] Let $F,H,A_1,A_2$ as in Theorem \ref{Suzuki}. Then the
conjugates of $F,H,A_1,A_2$ form a partition of $G$. 
In particular  $F,H,A_1,A_2$, their conjugates and the conjugates of their
characteristic subgroups  are $TI$ sets in $G$. 
\end{itemize}
\end{theorem}

\section{The combinatorial depth of subgroups of $Sz(q)$}

The main result of this paper is the following:

\begin{theorem}\label{main}
\begin{itemize}
\item [a)] Let us list the representatives of conjugacy classes of the maximal
subgroups of the Suzuki group $G=Sz(q)$. By Theorem \ref{Suzuki} they
are the following:
$N_G(F)$, $B_0=N_G(H)$, $B_1=N_G(A_1)$, $B_2=N_G(A_2)$ and $Sz(s)$ 
for maximal $s$ such that $s^t=q$ for some
positive integer $t>1$. The combinatorial
depths of these subgroups are:

$d_c(N_G(F),G)=5$, $d_c(B_0,G)=4$, $d_c(B_1,G)=4$, $d_c(B_2,G)=4$,
and  $d_c(Sz(s),G)=4$.
\item [b)] The following often used subgroups and their conjugates have combinatorial depth $3$:

\begin{itemize}
\item [-]  $F$ or    characteristic subgroups of $F$
\item [-]  subgroups of $H$
\item [-]  subgroups of $A_1$
\item [-]  subgroups of $A_2$
\item [-] subgroups of order $2$
\item [-] cyclic subgroups of order $4$
\item [-] $S_1\in Syl_2(Sz(s))$ for some Suzuki subgroup $Sz(s)$.
\end{itemize}
\item [c)] Some  $2$-subgroups (and their conjugates) have the following combinatorial depths:
\begin{itemize}
\item [-] the Klein four subgroups $K_4$ have $d_c(K_4,G)=4$
\item [-] $L\leq Z(F)$ of order $2^{f-1}$, where $|Z(F)|=2^f$ have
$d_c(L,G)=2f-2$.
\end{itemize}

\item [d)] $d_c(Sz(s),G)=4$ for any $s$ such that $s^t=q$ for some positive
integer $t>1$.
\item [e)] Viewing $G$ as a twisted group of Lie type, some subgroups of $G$
are of special importance. Using the notations of \cite{CSM} they are
$B^1=N_G(F)$, $H^1=H$, $N^1=N_G(H)$, $U^1=F$. Their depths can be read off
from the above lists. 
\end{itemize}
\end{theorem}

We will prove this theorem in a series of propositions and lemmas.

Note that by \cite[Thm.~3.12~(b),(c)]{BDK} every $TI$ subgroup of a group $G$
has combinatorial depth at most $3$ and has combinatorial
depth at most $2$ if and only if it is normal.

\begin{remark}\label{TI} It follows  that every nontrivial subgroup $L$ of $Sz(q)$
for $q\geq 8$ which is a $TI$ set has combinatorial depth $3$,  $\delta_*(L)=\delta(L)=\delta^*(L)=2$. In particular, by Theorem \ref{info} this holds for  $F,H,A_1,A_2$ and for the conjugates of
their nontrivial characteristic subgroups.
\end{remark}

From now on let $G=Sz(q)$, $G_1=Sz(s)$ a Suzuki subgroup of $G$.

\begin{proposition} Let $N=N_G(F)$. The combinatorial depth $d_c(N,G)$
of the $Hall$-subgroup $N$ is $5$,
$\delta_*(N)=\delta(N)=\delta^*(N)=3$.
\end{proposition}

\begin{proof} The subgroup  $N$ is a one-point stabilizer in the Zassenhaus group $G$.
Let $N^{x_1}$ and $N^{x_2}$ be two conjugates of $N$, which
are different from each other and from $N$. Then
$N\cap N^{x_1}\cap N^{x_2}$  stabilizes $3$ points, hence it is the identity.
However,  if $y_1$ does not normalize $N$ then  $N\cap N^{y_1}$ 
is a two point stabilizer in $G$, which is nontrivial.
Thus by Theorem \ref{BDK3.11} (b), $\delta_*(N)=\delta(N)=\delta^*(N)=3$.
Since $Core_G(N)=1\leq Z(G)$, it follows from Theorem \ref{BDK3.11} (c)  that 
$d_c(N,G)=2\delta(N)-1=5$.
\end{proof}

\noindent
\begin{lemma}\label{intersection}
Let $1\leq m_1,m_2\leq f$ be natural numbers. Then
two subgroups $H_1,H_2$ of $Z(F)$ of orders $2^{m_1}$ and $2^{m_2}$ respectively have intersection of order at least $2^{m_1+m_2-f}$, where $|Z(F)|=q=2^f$ $(f=2m+1)$.
\end{lemma}

\begin{proof}
This follows from the dimension formula for vector spaces,
(see Lemma 3.2 in \cite{Frit}).
\end{proof}

\begin{proposition}\label{F}
Let $L\leq Z(F)$ be a subgroup of order $2^{f-1}$, where $|Z(F)|=2^f$. Then
$d_c(L,G))=2f-2$. 
\end{proposition}

\begin{proof}
Our argument is similar to that of   \cite[Satz 3.3]{Frit}.
It is easy to see that the depth $\delta $ of ${\mathcal{U}}_{\infty}:={\mathcal{U}}_{\infty}(L,G)$ is at most $f$.
Since $N_G(F)$ acts transitively on the involutions of $F$,
and these generate $Z(F)$, we have that $Core_{N_G(F)}(L)=1$. 
By Lemma \ref{intersection} we need at least $f$ conjugates of $L$ by elements of $N_G(F)$, i.e.
$L_1,\ldots,L_f$ so that their intersection should be $1$. Hence $\delta (L)$  is also
$f$, and this is the same as $\delta_*(L)$, where conjugation is considered
by elements of $N_G(F)$. By Theorem \ref{BDK3.11} (c), $d_c(L,N_G(F))=2f-1$.
It means among other things that $L\cap L^{x_1}\cap \ldots L^{x_{f}}=
L\cap L^{y_1}\cap \ldots L^{y_{f-1}}$ for suitable elements $y_i$, where $i=1,\ldots, f-1$ and
${\mathcal{U}}_{f-2}\ne {\mathcal{U}}_{f-1}={\mathcal{U}}_{f}$.
We want to show that in $G$ the following holds: ${\mathcal{U}}_{f-3}\ne {\mathcal{U}}_{f-2}={\mathcal{U}}_{f-1}$.
Let us consider $L_{x_1,\ldots ,x_{f-1}}=L\cap L^{x_1}\cap \ldots L^{x_{f-1}}$.
We have to find elements $y_1,\ldots y_{f-2}$ in $G$ such that
$L_{y_1,\ldots y_{f-2}}=L_{x_1,\ldots x_{f-1}}$.  If $L_{x_1,\ldots ,x_{f-1}}=1$  then
$y_1$ can be chosen outside $N_G(F)$, since $F$ is $TI$, thus if
$y_1\in G\backslash N_G(F)$ then  $L\cap L^{y_1}=1$. 
If the intersection $L_{x_1,\ldots ,x_{f-1}}$ is of order  $2$
then if we intersect one by one, in some step we do not get a
smaller subgroup. So one of the $L^{x_i}$ can be cancelled. Further cancellation
is impossible
by Lemma \ref{intersection}, 
so the intersection cannot belong to ${\mathcal{U}}_{f-3}$.
However,  there are subgroups $L,L^{x_1}\ldots,L^{x_{f-1}}$ such that  their 
intersection is of order $2$,  since the depth of ${\mathcal{U}}_{\infty}(L,G)$ is $f$.
The depth $d_c(L,G)$ cannot be $2f-3$, for  if we consider the situation when
$L_{x_1,\ldots,x_{f-1}}$ is of order $2$ and  
$x_1=1$, then by Theorem \ref{BDK3.9} (ii) $y_1$ has
to centralize an involution in the intersection, thus $y_1\in F$ must hold.
Hence 
it centralizes $L$. But it is impossible since  by Lemma \ref{intersection},
a subgroup of order $2$ cannot be the intersection of less than $f-1$ subgroups.  
Thus $d_c(L,G)=2f-2$.
\end{proof}

The next result is about some other subgroups of $F$.

\begin{proposition} 
\begin{itemize}

\item [(i)] For every  subgroup $F_1\leq F$ of  order $2$ we have   $d_c(F_1,G)=3$.
\item [(ii)] We have  $d_c(Z(F),G)=3$.
\item [(iii)] Let $K_4$ be a subgroup of type $(2,2)$ in $F$.  Then $d_c(K_4,G)=4$.
\item [ (iv)] Let $C_4$ be a cyclic subgroup of order $4$ in $F$. Then
$d_c(C_4, G)=3$. 
\end{itemize}
\end{proposition}

\begin{proof}
\begin{itemize}
\item [(i)] We note that  $F_1$ is a $TI$ set, so by Remark \ref{TI} we are done.

\item [(ii)] Follows from Remark \ref{TI}. 

\item [(iii)]  First we note that $\delta_{*}(K_4)=2$, since if we take an element $x$ outside 
$N_G(F)$ then $K_4\cap K_4^x=1$.
According to Theorem \ref{BDK3.11} (a), $3=2\delta_*(K_4)-1\leq d_c(K_4,G)$.
We prove that the combinatorial depth of $K_4$ is at most $4$. 
By Theorem \ref{BDK3.9},  we have to show that
for every $x_1,x_2\in G$ there  exists an element $y_1$ such that 
$K_4\cap K_4^{x_1}\cap K_4^{x_2}=K_4\cap K_4^{y_1}$ ($*$).
If the intersection on  the left hand side of ($*$) is $1$ then we can choose any $y_1$ outside $N_G(F)$.
We mention that $N_G(K_4)=C_G(K_4)=F$ since $3$ does not divide the order of $G$.
So if the intersection on
the left hand side of ($*$) is of order $4$ then $x_1,x_2\in F$, so we may choose 
$y_1=x_1$.
If the intersection on the left hand side of ($*$) is of order $2$
then  $x_1,x_2\in N_G(F)$. If $x_1\not\in F$
then we may choose $y_1=x_1$.
If the left hand side of ($*$) is of order $2$ and $x_2\not\in F$ then we may choose
$y_1=x_2$.

If the combinatorial depth would be $3$ then by Theorem \ref{BDK3.9} (ii) 
for every $k\in {K_4}_{\{x_1,x_2\}}$,   $x_1kx_1^{-1}=y_1ky_1^{-1}$ 
should also hold.
We mention that there exists always an intersection $K_4\cap K_4^{x_1}\cap K_4^{x_2}$
which is of order $2$. Since if we take an element $x\in N_G(F)$ that takes an
involution of $K_4$ to another, then $x\not \in F$, so it cannot normalize $K_4$, but since $K_4\cap K_4^x\ne 1$, it can be only of order $2$.
Let us suppose now that the above intersection is of order $2$. Let the elements of $K_4$
be $1,e_1,e_2,e_3$. Let us suppose that $x_1\in F$ and $x_2\not\in F$, but  $x_2\in N_G(F)$. Let us suppose that $e_1^{x_2}=e_2$, and the above intersection is
$\langle e_2\rangle$. Then the above intersection is $K_4\cap K_4^{x_2}$.
This is centralized by $x_1$.
But if $y_1$ acts on the intersection as $x_1$ then it should also centralize
this intersection. However, we know that the centralizer of every nontrivival element
is in $F$, so $y_1\in F$. Thus it centralizes $K_4$ and the intersections
$K_4\cap K_4^{x_1}\cap K_4^{x_2}$ and
$K_4\cap K_4^{y_1}$ cannot be equal. Thus $d_c(K_4,G)=4$.
\item [(iv)]
We first prove that the depth is  at most $4$.
Let $C$ be a cyclic subgroup of $G$ of order $4$. We may assume that
$C\leq F$.
If $C_{x_1,x_2}$ is of order $4$  then let us take $y_1=x_1$. 
If the order is $1$ then let $y_1\in G\backslash F$.
If the order is  $2$ and
$C_{x_1}$ is of order $2$ then let $y_1=x_1$. If $C_{x_1}$ is of order $4$ and 
$C_{x_2}$ is of order $2$ then let $y_1=x_2$. We have that $C$ has depth at most $4$.
To prove that the depth is $3$, we  only have to deal with the case when the order of the intersection $C_{x_1,x_2}$ is $2$.
If $C_{x_1,x_2}$ is of order $2$,
$C_{x_1}$ is of order $4$ and $C_{x_2}$ is of order $2$ then $x_2$ centralizes
the unique  subgroup of $C$ of order $2$ and $x_1$ centralizes the whole $C_{x_1,x_2}$.

Thus
$x_2=y_1$ is a good choice, since on the  intersection of order $2$ both of them are acting trivially. Thus the depth of $C$ is $3$ in $G$.
\end{itemize}
\end{proof}

\begin{remark} Note that elementary abelian subgroups of type $(2,2,2)$
have combinatorial depth $3$ in $Sz(8)$, since these subgroups are  centers
of some Sylow $2$-subgroups of $Sz(8)$.
Calculations with GAP \cite{GAP4}  show that elementary abelian subgroups of type
$(2,2,2)$ have combinatorial depth $4$ in $Sz(32)$ and in $Sz(128)$.
\end{remark}

\begin{proposition}\label{ngai}
The combinatorial depth of the Frobenius
groups $B_i=N_G(A_i)$  in $G$ is $4$, for $i=1,2$.
\end{proposition}

\begin{proof} We prove the statement  for $i=1$, the proof for $i=2$  is similar.
First we prove that the depth is at most $4$.
If  $B_1\cap B_1^{x_1}\cap B_1^{x_2}=1$ then we  
want to prove that there exists an element $y_1\in G$ such that 
$B_1\cap B_1^{y_1}=1$.  
Let us suppose that $B_1\cap B_1^x\ne 1$. Since $N_G(A_1)=N_G(B_1)=B_1$, if
$x\not\in B_1$ then by Theorem \ref{info} b) we have that $B_1\cap B_1^x\cap A_1=1$.  Thus $B_1\cap B_1^x$
contains an involution $e^x$ such that $e\in B_1$. 
Since the involutions in $B_1$ are conjugate inside $B_1$, there is an element
$y\in B_1$ such that $e^x=e^y$. Then $xy^{-1}\in C_G(e)$ and $x\in C_G(e)B_1$.
Thus the number of elements $x$ such that $e\in B_1 $ lies in the conjugate
$B_1^x$ is $|C_G(e)B_1|=|C_G(e)A_1|\leq |F||A_1|=q^2|A_1|$.
In order to get an upper bound on the number of elements $x\in G$ such that
$B_1\cap  B_1^x $ is nontrivial, we have to multiply this by the number
of involutions in $B_1$, which is $|A_1|$. Thus the upper bound is at most
$q^2|A_1|^2$ and since $|A_1|< 2q$, we have that $q^2|A_1|^2<q^2(q^2+1)(q-1)=|G|$.
Thus there is an element $y_1$ such that
$B_1\cap B_1^{y_1}=1$.

If both of $x_i$ are in $B_1$, then $y_1=x_1$ is a good choice.
If some of the $x_i$ is not inside $B_1$, then
the intersection cannot contain elements of $A_1$,
since this is a $TI$-set. Then $B_{x_1,x_2}:=
B_1\cap B_1^{x_1}\cap B_1^{x_2}$
is either of order $4$ or of order $2$. 
If it is of order $4$ then either $B_1\cap B_1^{x_1}$
or $B_1\cap B_1^{x_2}$ is of order $4$.
We choose $y_1=x_1$ or $y_1=x_2$ in this case to get the order $4$ intersection.
If there is no intersection of order $4$,
then there must be an  intersection of order $2$.
If the intersection $B_{x_1,x_2}$ is of order $2$, 
then if $B_1\cap B_1^{x_1}$ is of order $2$ then we
choose $y_1=x_1$. If it is of order  $4$, 
then the $x_2$ cannot be  in $B_1$ and $B_1\cap B_1^{x_2}$ is of order $2$, Then choose $y_1=x_2$. 
If $B_1\cap B_1^{x_1}$ and $B_1\cap B_1^{x_2}$ are both of order  $4$, then their intersection
cannot be of order $2$, since in a Frobenius group two complements  have  trivial 
intersection.
Thus we have proved that the depth of $B_1$ is at most $4$.
Now we prove that the depth cannot be $3$.
Let $B_1\cap B_1^{x_2}=C$ be of order $4$. This happens  e.g. if $x_2$ is inside the
centre of the Sylow $2$-subgroup of $G$ containing $C$, but not inside $C$.
Let $x_1\in B_1$ such that $C^{x_1}\ne C$. We want to prove that there is no $y_1$ such that 
$B_1\cap B_1^{y_1}=C$ and $c^{x_1}=c^{y_1}$ for every $c\in C$.
Let us suppose by contradiction that such $y_1$ would exist. Then $C^{y_1}\ne C$.
There exists a subgroup $C_0\leq B_1$ such that $C_0^{y_1}=C$.
Let $c_0, c$ be  involutions of $C_0$ and $C$ respectively. Then 
$c_0^{y_1},c^{y_1}\in B_1$, since $c^{y_1}=c^{x_1}\in B_1$. Since $\langle c_0\rangle A_1$ is a Frobenius
group of index $2$ in $B_1$, $c\in \langle c_0\rangle A_1$, and
$cc_0\in A_1$. Thus similarly $(cc_0)^{y_1}=c^{y_1}c_0^{y_1}\in A_1$.
Since $A_1$ is a $TI$-set
and $y_1\not\in B_1$, this is a contradiction. Hence the depth of $B_1$ is $4$.
\end{proof}

\begin{proposition} The combinatorial depth of 
$B_0=N_G(H)$ is  $4$ in $G$.
\end{proposition}

\begin{proof} We know from  Theorem \ref{Suzuki} that $|B_0|=2(q-1)$.
First we prove that the depth of $B_0$ is at most $4$.
By Theorem \ref{info}, $H$ is also a $TI$-subgroup of $G$, so $B_0\cap B_0^{x_1}\cap B_0^{x_2}$ is of order $1$ or $2$, if
$x_1,x_2$ are not both in $B_0$. If it is of order $1$ then
we have to find an element $y_1$ such that $B_0\cap B_0^{y_1}=1$.
We use similar calculations as in Proposition \ref{ngai}.
If $x\not\in B_0=N_G(H)$ then
$B_0\cap B_0^x\cap H=1$. 
Thus if $B_0\cap B_0^x\ne 1$ then this intersection contains an
involution $e^x$ such that $e\in B_0$.
Since the involutions of $B_0$ are conjugate in $B_0$,
there exists an element $y\in B_0$ such that $e^x=e^y$.
Then  $xy^{-1}\in C_G(e)$ and $x\in C_G(e)B_0$.
Since  the  number of involutions in $B_0$ is $|H|=q-1$
and $|C_G(e)B_0|(q-1)=q^2(q-1)^2<|G|$, there exists an element $y_1\in G$ such that
$B_0\cap B_0^{y_1}=1$.

If ${B_0}_{\{x_1,x_2\}}$ is of order $2$ then either $B_0\cap B_0^{x_1}$ or $B_0\cap B_0^{x_2}$ is of order $2$.
In the first case again  let $y_1=x_1$. In the second case let $y_1=x_2$. 
We prove that the depth cannot be $3$.
Let $C:=B_0\cap B_0^{x_2}$ be of order $2$, e.g.  let $x_2\in Z(S)\backslash B_0$,
where $S\in Syl_2(G)$ intersecting $B_0$. Let $x_1\in B_0\backslash C$. 
Assume by contradiction that there exists an element $y_1\in G$ such that
$C^{x_1}=C^{y_1}$ and $B_0\cap B_0^{y_1}=C$.
Since $C^{x_1}\ne C$ thus $C^{y_1}\ne C$.
There exists a subgroup $C_0$ of order $2$ in $B_0$   such that  $C_0^{y_1}=C$ and
$C^{y_1}=C^{x_1}\leq B_0$. Let $c_0,c$ be the
unique involutions in $C$ and $C_0$, respectively. Then $c_0c\in H$ and $(c_0c)^{y_1}={c_0}^{y_1}c^{y_1}\in H$. 
However, $H$ is a $TI$-subgroup and $H\cap H^{{y_1}^{-1}}$ contains a
nontrivial element, so $y_1\in B_0$, which is a contradiction.
Thus the depth of $B_0$ cannot be $3$.
\end{proof}

\begin{remark} For finding an element $y\in G $ such that $B_0\cap B_0^y=1$
we also could have used a similar argument to that of Lemma 2.1 in \cite{Frit}.
\end{remark}

We now study the depth of a Suzuki subgroup $G_1$ of the Suzuki group $G$.

\begin{remark}\label{123}
First we remark that $N_{G}(G_1)=G_1$, since  $G$
does not contain a subgroup, where $G_1$ is a normal subgroup, by Theorem \ref{Suzuki}.
The intersections of the form $G_1\cap G_1^u$ have the following properties:
\begin{itemize}
\item [1.] If $G_1\cap G_1^u$ contains
a nontrivial cyclic subgroup $C$ of
odd order then it also  contains 
a maximal cyclic subgroup $L$ of $G_1$ such that $C\leq L$.

To see this,  by Theorem \ref{info}, 
$C$ is contained in a maximal cyclic subgroup $L$ of $G_1$ of odd order.
It is also contained in a conjugate of $L^u$ in $G_1^u$, say in $L^{uy}$.
Since $L$ is a characteristic subgroup of a $TI$-subgroup of $G$ it is also
$TI$-subgroup in $G$.  
But then $L=L^{uy}$, and so $L$ is contained in the intersection.

\item [2.] The intersection $G_1\cap G_1^u$ 
cannot be a  Suzuki group $Sz(s_1)$  for $s_1<s$.

To see this, we assume by contradiction that the intersection would be $Sz(s_1)$, where $s_1>2$.
Then there would be a maximal cyclic subgroup in it
of order $s_1-1$ and this would be contained in a maximal cyclic subgroup of 
$G_1$. Since $s_1|s$, this maximal cyclic
subgroup is of order $s-1$ which cannot be
contained in the smaller $Sz(s_1)$, since $s-1\not||Sz(s_1)|$ , which contradicts the fact that
by 1. a maximal cyclic subgroup of this order  must be contained in the intersection.
If $s_1=2$ then $|Sz(s_1)|=20$ and its cyclic subgroup of order $5$ is contained in
a maximal cyclic group of odd order of $G_1$, which is contained in the intersection
$G_1\cap G_1^u$ by part 1. Since the intersection is $Sz(s_1)$ this maximal cyclic subgroup of $G_1$
should be of order $5$, which cannot happen.

\item [3.] The intersection $G_1\cap G_1^u$  cannot be a dihedral
group.

To see this, let us suppose by contradiction that $G_1\cap G_1^u$ is a dihedral
group. Then it cannot be a $2$-group of order at least $8$, because a
Suzuki $2$-group does not contain such a subgroup. 
Let  $e\in G_1\cap G_1^u$ be an involution. Let $e\in S_1\in  Syl_2(G_1)$
and $e\in S_2\in Syl_2(G_1^u)$ thus $S_1\ne S_2$.
Then $S_1,S_2$ are conjugate in $G$. Since the Sylow $2$-subgroups of $G$ are $TI$ sets and $S_1$ intersects $S_2$ they must be inside the same
Sylow $2$-subgroup $S$ of $G$.
If $S_1^k=S_2$ then $k\in N_G(S)$.
Moreover $k$ cannot be in $N_G(Z(S_1))$ since $Z(S_1)$ is an elementary abelian
subgroup of order at least $8$, which cannot lie in a dihedral group. 
Also $k$ cannot be a $2$-element, since then it would centralize all the involutions
in $S$, and then $Z(S_1)$ would be in the intersection. If $k$ is a $2'$-element,
then it belongs to some complement $H$ in the Frobenius group $N_G(S)=HS$. This group
contains $N_G(Z(S_1))$.  Thus 
$N_G(Z(S_1))=(SH)\cap N_G(Z(S_1))=S(H\cap N_G(Z(S_1)))=SH_1$,
by the Dedekind identity (modularity). 
So the complement $H_1$ can be chosen as
a subgroup of $H$.
Let us take an involution
$e_1\in S_1$ such that $e_1^k=e$.
Since $N_{G_1}(S_1)\leq N_G(Z(S_1))$,
by the Dedekind identity $N_{G_1}(S_1)=H_2S_1$ and $H_2\leq H_1^s\leq H^s$ for some $s\in S$.
Since $H_2$ is sharply $1$-transitive on
the involutions of $S_1$, there is an element $l\in H_2$ such that $e_1^l=e$.
However, since $e_1^k=e$ thus $e_1^{k^s}=e$ and $H^s$ is sharply $1$-transitive on the involutions of $Z(S)$, we have that $l=k^s$.
This is a contradiction, since $k^s$ does not normalize $Z(S_1)$, however
$l\in N_G(Z(S_1))$.
\end{itemize}
\end{remark}

\begin{proposition}\label{normalizer} Let $q=2^{2m+1}$, let $s=2^{2p+1}$ with $s^t=q$, let
$G_1=Sz(s)<G=Sz(q)$.
. Let $\pi\in Aut(GF(q))$  such that $\pi: x\mapsto x^{2^{m+1}}$.
Then the restriction of $\pi $ to $GF(s)$ is the automorphism of $GF(s)$
mapping $x$ to $x^{2^{p+1}}$.
Let \[u(a,b):=
\begin{pmatrix} 1 & 0 & 0 \\
	       a & 1 & 0\\
	       b & a\pi & 1
\end{pmatrix}\]
Let us suppose that $S:=\{u(a,b)\ a,b\in GF(q)\}$ and $S_1=\{u(a,b)| a,b\in GF(s)\}$.
(So  $S_1\in Syl_2(G_1)$,  $S\in Syl_2(G)$ and  $S_1< S$.)

Then $N_S(Z(S_1)u(1,0))=\{ u(c,d) |c,d\in GF(q), c+c\pi \in GF(s) \}=\{ u(c,d)|c\in GF(s), d\in GF(q)\}$
and
$N_S(S_1)=\{ u(c,d)| c,d \in GF(q), a(c\pi )+ c(a\pi)\in GF(s)$
${\rm for}$ $ {\rm every}$ $a \in GF(s)\}$.
Furthermore, $N_S(Z(S_1)u(1,0))=N_S(S_1)$ and $S>N_S(S_1)$.
Moreover, $N_S(S_1)=N_S(Z(S_1)x)$
also holds  for every nontrivial coset $Z(S_1)x$ in $S_1$.
\end{proposition}

\begin{proof} The restriction of $\pi $ to $GF(s)$ is an automorphism
of it having the property that the square of it raises every element to 
the second power. Since such an automorphism is unique it has to act 
the way as it is written in the Proposition.

Since multiplication in $S$ is defined by $u(a,b)u(c,d)=u(a+c,c(a\pi )+b+d)$, thus
$u(a,b)^{-1}=u(a,b+a(a\pi ))$ and 
$Z(S_1)=\{ u(0,b)|b\in GF(s) \}$.  Furthermore $u(0,b)u(1,0)=u(1,b)$.
Similarly
$u(c,d+c(c\pi))u(1,b)u(c,d)=u(1,b+c+(c\pi))$, hence $N_S(Z(S_1)u(1,0))=\{
u(c,d)|c,d\in GF(q), c+(c\pi)\in GF(s) \}$.
To calculate the normalizer of $S_1$ we note that
$u(a,b)\in S_1$ iff $a,b\in GF(s)$.
We have to calculate when $u(c,d+c(c\pi))u(a,b)u(c,d)=
u(c+a,d+c(c\pi)+b+c(a\pi))u(c,d)=u(a,d+b+c(c\pi )+c(a\pi )+d+(c+a)(c\pi ))=u(a,b+c(a\pi )+a(c\pi ))$
belongs to $S_1$. 
This happens if and only if $c(a\pi )+a(c\pi )\in GF(s)$
for every $a\in GF(s)$, 
thus $N_S(S_1)=\{u(c,d)|c,d\in GF(q),c(a\pi )+a(c\pi) \in GF(s)$ 
${\rm for}$ ${\rm every}$ $ a\in GF(s) \}$.

Obviously $N_S(S_1)\leq N_S(Z(S_1)u(1,0))$. We want to prove that if  $c\in GF(q)$
and 
$c+(c\pi)\in GF(s)$ then $c\in GF(s)$. This would imply that the inequality 
cannot be strict. The map 
$id+\pi $ is  $GF(2)$-linear  from $GF(q)$ into itself.
Since $\pi $ has no nontrivial fixed points, the kernel of $id+\pi$ is $GF(2)$.
This map is not surjective, the identity of $GF(q)$ is not in the image of
$id+\pi $. For
if $x\pi +x=1$ then $x\pi =x+1$  and $x^2=x\pi^2 = x+2=x$ and thus $x\in \{0,1\}$.
Since the image of $0$ and $1$  under $id+\pi$ is zero, $1$ is not in  the 
image
of $id+\pi $.
Thus $GF(q)=Im(id+\pi)\oplus GF(2)$ and similarly
$GF(s)=Im(id+\pi)_{GF(s)}\oplus GF(2)$. 
However, $Im(id+\pi)_{GF(s)}=Im(id+\pi)\cap GF(s)$ by the Dedekind identity.
Let $c\in GF(q)\backslash GF(s)$. Then if $c(id +\pi) $ would  be in $GF(s)$
then there would be an element $d\in GF(s)$ such that $d(id+\pi)=c(id+\pi)$.
But then $d-c\in  GF(2)$, so $d\in GF(s)$, which is a contradiction. 

It is easy to see that $S>N_S(S_1)$, since $|S|=q^2$, however $|N_S(S_1)|=qs$.
If $Z(S_1)x$ is another nontrivial coset of $Z(S_1)$ in $S_1$ then
since by \cite[Thm.~6.8 in  Ch.~VIII]{H-Bl_II}, $N_{G_1}(S_1)$ acts transitively on 
cosets of $S_1/Z(S_1)$, there is an element $h\in N_{G_1}(S_1)$
such that $(Z(S_1)u(1,0))^h=Z(S_1)x$. Then $N_S(S_1)^h=N_S(Z(S_1)u(1,0))^h=N_S(Z(S_1)x)$. However, $N_{G_1}(S_1)\leq N_G(S)$, so it normalizes $N_S(S_1)$.
Thus $N_S(Z(S_1)x)=N_S(S_1)^h=N_S(S_1)$. So we are done.
\end{proof}

\begin{proposition}\label{S_1cap S_1^x} The intersection of two conjugates of $S_1\in Syl_2(G_1)$ in
$G$ 
can be either $1$ or $S_1$ or $Z(S_1)$, and all these intersections occur.
\end{proposition}

\begin{proof} 
Since $S_1<S\in Syl_2(G)$ and $S$ is $TI$, if we take $u$ outside 
the normalizer of $S$, then $S_1\cap S_1^u=1$.
Let $u\in N_{G}(S_1)$. Then $S_1\cap S_1^u=S_1$.

Let us suppose that  $S_1\in Syl_2(G_1)$ and $S_1<S\in Syl_2(G)$,
$1\ne S_1\cap S_1^u=Z<Z(S_1)$.  Then $u\in N_G(S)$. In fact $u$ is of odd order, and it belongs to a complement $K$ in the Frobenius group $N_G(S)$.
By the Dedekind identity the complement $K_1$ of the subgroup $N_G(Z(S_1))=K_1S$
can be chosen to be a subgroup of $K$.
Let $1\ne x\in Z$. Then
$o(x)=2$. Thus there exists an element $y\in Z(S_1)$ such that $y^u=x$. 
Since the Frobenius complement $K_2$ of $N_{G_1}(S_1)=K_2S$ acts sharply $1$-transitively
on the involutions of $S_1$,
there exists an  element   $n\in K_2$ such that 
$y^n=x$. 
There is an element $s\in S$ such that
$n_1=n^s\in K_1\leq K$. Then $y^{n_1}=x$ and $y^u=x$. Thus $u{n_1}^{-1}\in C_G(y)\leq S$. On the other hand it belongs to $K$.
Hence $u=n_1$. However, $u=n_1\in N_G(Z(S_1))$ and hence $S_1\cap S_1^u\geq Z(S_1)$,
which contradicts to our assumption.

Thus the intersection cannot be a nontrivial  proper subgroup of $Z(S_1)$.

By the previous argument we also have that
if $S_1\cap S_1^u\ne 1$ then $S_1\cap S_1^u\geq Z(S_1)$. Thus $u\in N_G(S)$.
Now we want to show that if $S_1\cap S_1^u>Z(S_1)$ then $S_1\cap S_1^u=S_1$.

Let us suppose now that the $Z(S_1)<S_1\cap S_1^u =L<S_1$.
First we note that $Z(S_1)\leq Z(S_1)^u=Z(S_1^u)<S_1^u$, but since their order is the same $Z(S_1)=Z(S_1)^u$.
Let $Z(S_1)x\in S_1\cap S_1^u$, where $o(x)=4$.
Then since $u^{-1}$ fixes $Z(S_1)$ it takes the coset of $x$ to another coset
in $S_1$, say  ${Z(S_1)x}^{u^{-1}}=Z(S_1)y$,
where $y\in S_1$ has  order $4$.
By \cite[Thm.~6.8 in Ch.~VIII.]{H-Bl_II}
the complement of $N_{G}(S)$ acts sharply $1$-transitively on $S/Z(S)$ and
the complement of $N_{G_1}(S_1)$ acts sharply $1$-transitively on $S_1/Z(S_1)$.
Thus there is an element $h\in N_{G_1}(S_1)\leq N_G(S)$ such that $(Z(S_1)x)^h=Z(S_1)y$.
But then $hu\in N_G(Z(S_1)x)$ and hence $hu\in N_G(Z(S)x)$. However, since $hu\in N_G(S)$ and the elements of the complement 
act sharply $1$-transitively  on $S/Z(S)$, 
$hu\in S$. Thus $hu\in N_{S}(Z(S_1)x)=N_S(S_1)$  
by Proposition \ref{normalizer}. Thus $u$ normalizes $S_1$, which is a contradiction.

The intersection can also be $Z(S_1)$, e.g. let $u\in S\backslash N_S(S_1)$, then
$S_1>S_1\cap S_1^u\geq Z(S_1)$, so the intersection is $Z(S_1)$.
\end{proof}

\begin{corollary}\label{dcS_1} Let $S_1\in Syl_2(G_1)$, where $G_1=Sz(s)$ is a Suzuki subgroup
of $G=Sz(q)$, where $s^t=q$ and $q=2^{2m+1}$. Then the combinatorial depth $d_c(S_1,G)=3$.
\end{corollary}

\begin{proof} Since $S_1$ is not normal in $G$, $d_c(S_1,G)>2$.
If $S_1\cap S_1^{x_1}\cap S_1^{x_2}=1$ then $y_1$ can be chosen outside $N_G(S)$, for
$S_1<S\in Syl_2(G)$.
If $S_1\cap S_1^{x_1}\cap S_1^{x_2}=S_1$ then $y_1=x_1$ is a good choice.
If $S_1\cap S_1^{x_1}\cap S_1^{x_2}=Z(S_1)$ then if $S_1\cap S_1^{x_1}=Z(S_1)$
then
we choose $y_1=x_1$ if $S_1\cap S_1^{x_1}=S_1$ and $x_1$ is a $2$-element in $N_G(S_1)$, then it centralizes $Z(S_1)$, so  we choose 
$y_1\in S\backslash N_S(S_1)$. This is possible since $S>N_S(S_1)$, by Proposition \ref{normalizer}.
On the other hand, if $x_1$ is an element of odd order in $N_G(S_1)=K_1S_2\leq N_G(S)=KS$, then
we observe that
$N_G(Z(S_1))=\tilde K_1S$, where $|K_1|=|\tilde K_1|=|Z(S_1)|-1$, by the 
sharply $1$-transitivity of the action on the involutions
by $K_1$ and $\tilde K_1$. We may suppose
that $x_1\in K_1$. Let us choose $y_1$
in a complement $\tilde K_1$  of $N_G(Z(S_1))$
such that $\tilde K_1\not\leq N_G(S_1)$, and $y_1$ is conjugate to $x_1$ in $N_G(Z(S_1))$. Such a complement $\tilde K_1 $ exists, since $N_G(Z(S_1))$ has more complements
than $N_G(S_1)$. Then $y_1={x_1}^s$, where
$s\in S$, thus the action of $y_1$ and $x_1$ 
on $Z(S_1)$ is the same. However $S_1\cap {S_1}^{y_1}=Z(S_1)$, since  this
intersection contains $Z(S_1)$ and does not contain $S_1$. Thus by Proposition 
\ref{S_1cap S_1^x} gives us that the intersection is $Z(S_1)$.

Now we show that  $1<{S_1}_{x_1,x_2}=Z<Z(S_1)$ cannot hold.
Let us suppose now that $1\ne S_1\cap S_1^{x_1}\cap S_1^{x_2}=Z<Z(S_1)$.
Then $S_1\cap S_1^{x_1}=Z(S_1)$ must hold  and $Z(S_1)\cap S_1^{x_2}=Z$.
Then $1<Z(S_1)^{x_2^{-1}}\cap S_1\leq S_1^{x_2^{-1}}\cap S_1$ and since it is a proper subgroup of $S_1$,
by  Proposition \ref{S_1cap S_1^x},
$S_1^{x_2^{-1}}\cap S_1=Z(S_1)$.
Thus $S_1\cap S_1^{x_2}=Z(S_1)^{x_2}=Z(S_1)$ which cannot be the case,
since ${S_1}_{\{x_1,x_2\}}$ is 
a proper subgroup of $Z(S_1)$.
Thus the above intersection cannot be $1\ne Z<Z(S_1)$.
It is easy to see that the chosen
elements $y_1$ in each case act the same way on
${S_1}_{\{x_1,x_2\}}$
as $x_1$. Thus
$d_c(S_1,G)=3$.
\end{proof}

\begin{corollary} Let the notations be as in Corollary \ref{dcS_1}.
If  the intersection $G_1\cap {G_1}^x$ contains $S_1$ 
then this intersection cannot contain an  element of odd order,
except for the case when
this intersection is $G_1$.
\end{corollary}

\begin{proof} Let $G_1\cap G_1^x\geq S_1$. Then
$S_1,S_1^{x^{-1}}\leq G_1$. Thus there exists an element $g_1\in G_1$  such that
$S_1^{g_1}=S_1^{x^{-1}}$ and so $S_1^{g_1x}=S_1$ and  $n=g_1x\in N_G(S_1)$. Thus
${G_1}^x={G_1}^n$. So we may assume that  already at the beginning $x\in N_G(S_1)$.
If the intersection $G_1\cap G_1^x$ would contain an 
element of odd order, then  by Remark \ref{123}/1, it would also contain a maximal cyclic subgroup 
of odd order in $G_1$. Let us denote such a maximal cyclic subgroup  by $K_1$.
If $K_1\not\leq N_{G_1}(S_1)$ then $\langle K_1,S_1 \rangle=G_1$,
since no proper  subgroup of $G_1$
contains both $K_1$ and $S_1$.  Thus if the intersection is a proper subgroup of $G_1$
then $K_1\leq N_{G_1}(S_1)$ and $N_{G_1}(S_1)=K_1S_1$.
So we may assume that $G_1\cap G_1^x=K_1S_1$.
We know that $x\in N_G(S_1)$. If $x$ is a $2$-element, then since 
$K_1^{x^{-1}}\leq N_G(S_1)$
and $K_1^{x^{-1}}\leq G_1$, thus $K_1^{x^{-1}}\leq N_{G_1}(S_1)=K_1S_1$. Hence we have that 
${K_1}^{x^{-1}}={K_1}^{s_1}$ for some $s_1\in S_1$.
Thus $K_1=K_1^{s_1x}$. Since $x\in N_{G_1}(S_1)\leq N_G(S)$,
both $s_1$ and $x$ belong to $S$, so their product is
an element of $S$ normalizing $K_1$, which is impossible, since $N_G(S)$ is a Frobenius group.
Let now $x$ be of odd order in $N_G(S_1)$. We have that
$K_1^x\leq N_G(S_1)$ and $K_1^x\leq G_1^x$.
Since $K_1,K_1^x\leq N_{G_1^x}(S_1)$, they are conjugate, so there is an element
$y\in N_{G_1^x}(S_1)$ such that  $K_1=K_1^{xy}$. Thus $xy\in N_G(K_1)$. 
However $xy\in N_G(S_1)\leq N_G(S)$,
so it belongs to a complement $K_2$ of $N_G(S_1)$
containing $K_1$. Since $K_1$ acts sharply $1$-transitively on the involutions of $S_1$, $K_2$ cannot be bigger than $K_1$, otherwise some involution would be centralized
by a $2'$-element, which is not possible in a Suzuki group.
Hence $xy\in K_1$. But $x\not\in G_1^x, y \in G_1^x$ thus
$xy\not\in G_1^x$, contradicting the fact that $K_1\leq G_1^x$.
Thus $G_1\cap G_1^x$ is a $2$-group.
\end{proof}

\begin{proposition}\label{G_1capG_1^x} We use the notations of Corollary \ref{dcS_1}. 
If $G_1\cap G_1^x $ 
contains an involution $i$, then this intersection is either
$Z(S_1)$ or $S_1$ or $G_1$, where $i\in S_1\in Syl_2(G_1)$.
\end{proposition}

\begin{proof} 
Let us suppose that the 
intersection $G_1\cap G_1^x$  contains the involution $i$ and this
intersection is a proper subgroup of $G_1$.
Let $S_2\in Syl_2(G_1^x)$ containing $i$.
Then $S_1,S_2$ are contained in the same Sylow $2$-subgroup $S$ of $G$ and
they are conjugate by an element  $k\in N_G(S)$.
If $k\in S$, then $k$ centralizes $Z(S_1)$, and hence $Z(S_1)=Z(S_1)^k=Z(S_2)$ and 
thus $Z(S_1)\leq G_1\cap G_1^x$.
If $k$ is of odd order, then let us take an involution $j\in S_1$  such that $j^k=i$.
The subgroup $N_G(Z(S_1))$ is contained in $N_G(S)=SK$. We may suppose that $k$ belongs to the complement $K$.  
By the Dedekind identity $N_G(Z(S_1))=S(K\cap N_G(Z(S_1)):=SK_2$.
However, $N_{G_1}(S_1)=S_1K_1\leq N_G(S_1)\leq  N_G(Z(S_1))$, thus there exists an element $s\in S$ such that ${K_1}^s\leq K_2$. 
We note that $K_1$ acts sharply $1$-transitively on the involutions of $Z(S_1)$, and $K_1^s$
is acting the same way, since elements of $S$ centralize $Z(S_1)$.
We  will show that $K_2=K_1^s$. An element in
$K_2\backslash K_1^s$ would take an involution to the same
place as an element in $K_1^s$ and their quotient would centralize that involution, 
which cannot happen in a Frobenius  group. Thus $N_G(Z(S_1))=SK_1^s$
and there is an element $h\in K_1^s$ such that
$j^h=i$. But then $hk^{-1}$ is in $K$ and centralizes $i$, which can
only happen if $h=k$, and hence $k\in N_G(Z(S_1))$.
Thus, $Z(S_1)=Z(S_1)^k\leq S_1^k=S_2$
and hence $G_1\cap {G_1}^x\geq Z(S_1)$.

We know that $Z(S_1),Z(S_1)^{x^{-1}}\leq G_1$.
So there is an element $g_1\in G_1$ such that $Z(S_1)^{g_1}=Z(S_1)^{x^{-1}}$. 
Thus, $Z(S_1)^{g_1x}=Z(S_1)$ and $g_1x\in N_{G_1}(Z(S_1))$. So we may assume that
already originally $x$  belongs to $N_G(Z(S_1))$. 
If $G_1\cap G_1^x$ would contain an element of odd order then by Remark \ref{123}/1,
it would also contain
a maximal cyclic subgroup
$H_1\leq G_1$ of odd order. Because
of the subgroup structure of $G_1$, 
if $x$ does not normalize $G_1$ then $H_1$ normalizes $Z(S_1)$.
Similarly, $H_1,H_1^x\leq N_{G_1^x}(Z(S_1))\leq N_G(S)$.
Hence $H_1,H_1^x$ are conjugate in the Frobenius group $N_{G_1^x}(Z(S_1))$, i.e.
there is an element $y\in N_{G_1^x}(Z(S_1))$ such that $H_1^{xy}=H_1$. Thus
$xy\in N_G(H_1)$. However, $xy\in N_G(Z(S_1))$.
If $xy$ is an involution then it belongs to $S$, and then it cannot normalize $H_1$.
If $xy$ is of odd order then  it belongs to $H_1$. However, $x\not\in G_1^x$ and
$y\in G_1^x$, thus $xy\not\in G_1^x$. This contradicts  the fact that $H_1\leq G_1^x$.

Thus if $G_1\cap G_1^x$ contains an involution then it is either $Z(S_1)$ or $S_1$, since
it cannot contain an element of odd order.
The intersection can be $Z(S_1)$  e.g. when $S_1\cap S_1^x=Z(S_1)$. 
It happens e.g. when $x\in S\in Syl_2(G)$ containing $S_1$, but not normalizing
$S_1$. Such an element exists, since  $S>N_S(S_1)$ holds by Proposition \ref{normalizer}. If for some
$S_2\in Syl_2(G_1)$ the equality $S_2^x=S_1$ would hold, then $S_2\leq S$ and since $G_1\cap S=S_1$,
$S_1=S_2$, which is not the case.  

The intersection is $S_1$
if $x$ normalizes $S_1$ but it does not normalize $G_1$.
We can get such an element from $Z(S)$ since it cannot normalize $G_1$.
To see, this, let us use \cite[3.12 Remarks c) in Ch.~XI]{H-Bl_III}, telling
that 
the outer automorphism group of $G_1$ is cyclic of odd order. We have that no element of $Z(S)$
centralizes $G_1$, since the centralizers of $2$-elements are $2$-groups.
If $Z(S)\leq N_G(G_1)$
then then  $Z(S)/Z(S_1)\leq Out(G_1)$, which is a contradiction. 
\end{proof}

\begin{proposition}\label{occur} We use notations of Cor. \ref{dcS_1}. The intersection $G_1\cap G_1^x$ can be up to conjugacy
either $1$, $Z(S_1)$, $S_1$, $K_1$, ${A_{1_1}}$, ${A_{2_1}}$ or $G_1$, where
$K_1$ is maximal cyclic of order $s-1$, $A_{1_1},A_{2_1}$ are maximal cyclic of orders 
$s+2r_1+1$ and $s-2r_1+1$, respectively, where $s=2^{2p+1}$, $r_1=2^p$ and $q=s^t$.
Moreover, these subgroups occur as intersections.
\end{proposition}

\begin{proof} If $G_1\cap G_1^x\ne G_1$ contains a maximal cyclic subgroup of $G_1$ of odd order then it cannot contain other maximal cyclic
subgroups, because of the subgroup structure.
In fact one can get such an intersection if we take a cyclic subgroup $K_1\leq G_1$
of  order $s-1$. Then it is contained in a maximal cyclic subgroup $K$ of order $q-1$
in $G$. Then $K\not\leq N_G(G_1)$, since otherwise  $G_1$ would contain all the involutions
of $S\in Syl_2(G)$ normalized by $K$, which is not the case. 
Let $x\in K\backslash N_G(G_1)$. Then $G_1\cap G_1^x=K_1$.
The subgroup $A_{1_1}$ (and similarly $A_{2_1}$) is contained in a  maximal cyclic subgroup of $G$ of
odd order. It is definitely bigger than $|A_{1_1}|$. Let us take an element
$x$ of such  a maximal cyclic subgroup of odd order in $G$ outside $G_1$, then
$G_1\cap G_1^x$ contains $A_{1_1}$. It cannot normalize $G_1$, since then together with $G_1$ it would generate a group, which cannot be a subgroup of $G$ by Theorem \ref{Suzuki}. 
Thus the intersection 
is equal  to $A_{1_1}$. Similarly one can contstruct the intersection to be $A_{2_1}$.
The fact that $G_1\cap {G_1}^x=1$ occurs will be shown later in Proposition \ref{one}. 
The
rest follows from  Proposition \ref{G_1capG_1^x}.
\end{proof}

\begin{corollary}\label{dcG_1} With the notations of Proposition \ref{occur}, the intersection $G_1\cap G_1^{x_1}\cap G_1^{x_2}$ can
be $1$, $Z(S_1)$, $S_1$, $K_1$, $A_{1_1}$, $A_{2_1}$, $G_1$ or some conjugate subgroups
of $G_1$. The combinatorial depth  $d_c(G_1,G)$ is $4$. 
\end{corollary}

\begin{proof}
These subgroups occur as ${G_1}_{\{x_1,x_2\}}$, since by Proposition \ref{occur} they occur as $G_1\cap {G_1}^{x_1}$.
We have to show that there are no more possible triple intersections.
However, these would be some subgroups of groups of the above types. On the other 
hand,
$G_1\cap G_1^{x_1}$ is one of the above types, hence  $G_1\cap {G_1}^{x_2}$
is either disjoint to it, or  also one of these, hence their intersection is either
$1$ or one from the above list.

The combinatorial depth of $G_1$ cannot be $2$ since it is not normal in $G$.
We want to show that its combinatorial depth is 
at  most $4$. To see this, we observe that each subgroup which can be 
${G_1}_{\{x_1,x_2\}}$ is already $G_1\cap {G_1}^y$ for some suitable $y\in G$. 
This comes from  
Proposition \ref{occur}  and the first part of this Corollary.

We now show that $d_c(G_1,G)\geq 4$.
Let $G_1\cap G_1^{x_1}\cap G_1^{x_2}=A_{1_1}$.
Let us suppose that $A_{1_1}\ne A_{1_1}^{x_1}\leq G_1$. This happens e.g. when $x_1\in
G_1\backslash N_{G_1}(A_{1_1})$ and $G_1\cap {G_1}^{x_2}=A_{1_1}$.
If the combinatorial depth of $G_1$ in $G$ would be $3$,
then there would exist an element $y_1\in G$ such that 
$G_1\cap G_1^{y_1}=A_{1_1}$ and $A_{1_1}^{x_1}=A_{1_1}^{y_1}$.  However,
$A_{1_1}^{y_1}\leq G_1$,
so $A_{1_1}^{y_1}\leq G_1\cap G_1^{y_1}=A_{1_1}$, which is a contradiction.
Thus $d_c(G_1,G)$ is at least $4$.
\end{proof}

\begin{remark} We observe that $N_G(G_1)=G_1$, since in 
a Suzuki group there is no 
larger subgroup where a smaller Suzuki group would be a normal subgroup.
\end{remark}

\begin{proposition}\label{Z(S_1)} Let $G_1, S_1$ be as in Cor.~\ref{dcS_1}.
\begin{itemize}
\item [a)]
The set $\{ x\in G | G_1>G_1\cap G_1^x\geq Z(S_1)^u
,$ $u\in G \}$ 
has size  at most $(s-1)(q^2-s^2)(s^2+1)^2$
\item [b)]
With the notations of Prop. \ref{occur}, the size of the set
$\{x\in G | G_1\cap G_1^x=K_1^u,$ $ u \in G\}$ is
at most ${{(q-s)s^4(s^2+1)^2}\over 2}$.
\item [c)]
With the  notations of Prop. \ref{occur}, the sets 
$\{x\in G| G_1\cap G_1^x=A_{1_1}^u,$ $u \in G\}$ 
and $\{x\in G|G_1\cap G_1^x=A_{2_1}^u,$ $u \in G\}$ have sizes
${{s^4(s-1)^2}\over 4}(s-2r_1+1)^2(q-s+2(r-r_1))$ and
${{s^4(s-1)^2)}\over 4}(s+2r_1+1)^2(q-s-2(r-r_1))$, respectively.
\end{itemize}
\end{proposition}

\begin{proof}
\begin{itemize}
\item[a)]
If the intersection $G_1\cap {G_1}^x$ contains $Z(S_1)$ then there must be
$S_2\in Syl_2(G_1)$ such that  $S_2^x=S_1$ and so $Z(S_2)^x=Z(S_1)$.
By Sylow's theorem there exists an element $g_1\in G_1$
such that $Z(S_1)^{g_1}=Z(S_2)$ and so $Z(S_1)^{g_1x}=Z(S_1)$. Thus such an element $x$ can
be chosen from cosets of $G_1$ represented by some element of $N_G(Z(S_1))$. 
Two such cosets are different if the representing elements
are in different cosets of $N_{G_1}(Z(S_1))$.
Since $Z(S_1)^u$ must be a subgroup of $G_1$, the conjugating element can be chosen
from $G_1$.
Thus the set $\{ x\in G| G_1>G_1\cap G_1\geq Z(S_1)^u,$ $u \in G \}$
has at most $(|N_G(Z(S_1)):N_{G_1}(Z(S_1))|-1)|G_1||G_1:N_{G_1}(Z(S_1))|$ elements.
We have seen in the proof of Prop.~\ref{G_1capG_1^x}  that $|N_G(Z(S_1))|=(s-1)q^2, |N_{G_1}(Z(S_1))|=(s-1)s^2$, and the above product
is $({{(s-1)q^2}\over {(s-1)s^2}}-1)s^2(s^2+1)(s-1)(s^2+1)$. This is exactly what was stated.
\item [b)]
By similar arguments as in part a), we have to calculate the number
$|N_G(K_1):N_{G_1}(K_1)|-1)|G_1||G_1:N_{G_1}(K_1)|=({{2(q-1)}\over {2(s-1)}}-1)s^2(s^2+1)(s-1)s^2(s^2+1)/2$. This is exactly what  was stated.
\item [c)]
Similar to the proofs of parts a) and b).
\end{itemize}
\end{proof}

\begin{proposition}\label{one} We use notations of Prop \ref{occur}.
There exists an element $x\in G$ such that $G_1\cap G_1^x=1$.
\end{proposition}

\begin{proof} We  have to show that the complement set of those elements $x\in G$ for
which $G_1\cap G_1^x\ne 1$ is nonempty. We use the results of the previous propositions.
We note that  $q=s^t$. We have to prove that
$s^{2t}(s^t-1)(s^{2t}+1)> (s-1)(s^2+1)^2(s^{2t}-s^2)+s^4(s^2+1)^2(s^t-s)/2+
s^4(s-1)^2(s-2r_1+1)^2(s^t-s+2(r-r_1))/4+s^4(s-1)^2(s+2r_1+1)^2(s^t-s-2(r-r_1))/4+|G_1|$.
Since $s\geq 8$, if $t=3$ then the first part of the right hand side is bigger than
than any other part on the right hand side. The left hand side is bigger than the first part of the 
right hand side  multiplied by $5$, so we are done. If $t\geq 5$ then
the  $s^4$ times of the first part of the right hand side is  bigger than any other part on the right hand side, so it is enough
to prove that
$s^{2t}(s^t-1)(s^{2t}+1)> 5s^4(s-1)(s^2+1)^2(s^{2t}-s^2)$ which is obviously true.
So we are done.
\end{proof}

\section {The ordinary depth   of subgroups of $Sz(q)$}

As a consequece of the results of the previous section we have the following: 

\begin{corollary}
The ordinary depth of all subgroups of $G=Sz(q)$ mentioned in 
Theorem \ref{main} 
is $3$ except for $d(N_G(F),G)$ which is $5$.
\end{corollary}

\begin{proof}
 Since $G$ is simple, the depth
of each nontrivial subgroup is at least $3$.
Since for subgroups $B_0,B_1,B_2$ and for proper Suzuki subgroups $Sz(s)$
there exist  conjugates, which intersect them trivially, these subgroups have
ordinary depth $3$ by Theorem \ref{inter}.
The subgroups in  part  b) of Theorem \ref{main}  have combinatorial depth $3$, so they also have ordinary depth $3$. 
Since $F$ is $TI$,  so by Theorem \ref{inter},
 the ordinary depth of each $2$-subgroup (including those in part c)) is  $3$. 
The maximal subgroup $N_G(F)$ has ordinary depth at most $5$ since $d(N_G(F),G)\leq d_c(N_G(F),G)$.
If $d(N_G(F),G)\leq 4$ then by Theorem \ref{distance}, $m(\chi)\leq 1$ for
each $\chi\in {\irr}(G)$. However $m(1_G)\geq 2$, since $1_{N_G(F)}$ is irreducible
and if $\psi\in {\irr}(N_G(F))$ containing $F$ in its kernel, then
$d(\psi,1_{N_G(F)})=2$.
 Here we used \cite[Thm 5.9, Thm 5.10 in Ch XI]{H-Bl_III}.
 Thus $d(N_G(F),G)=5$.

\end{proof}
\noindent

{\bf Acknowledgements}
Research supported by National Scientific Research Grant No. K77476 and 
T\'AMOP-4.2.2.B-10/1-2010-0009.
    The authors thank Thomas Breuer for several useful suggestions and comments.\\


\bibliographystyle{amsalpha}

\end{document}